\documentclass[12pt]{amsart}
\usepackage{amssymb,epsfig,amsmath,latexsym,amsthm}
\usepackage{graphicx}
\usepackage{amsgen}
\usepackage[all]{xy}
\usepackage{psfrag}
\usepackage[T1]{fontenc}
\usepackage{color, xcolor, colortbl}
\usepackage{float}

\usepackage{comment} 

\theoremstyle{plain}
\newtheorem{theorem}{Theorem}[section]
\newtheorem{lemma}[theorem]{Lemma}

\newtheorem{proposition}[theorem]{Proposition}
\newtheorem{corollary}[theorem]{Corollary}

\theoremstyle{definition}
\newtheorem{definition}[theorem]{Definition}
\newtheorem{definition/construction}[theorem]{Definition/Construction}

\newtheorem{example}[theorem]{Example}

\newcommand{\Zed}{\mathbb Z} 
\newcommand{\Real}{{\mathbb R }}

\newcommand{\rank}{\mbox{\rm rank\,}}
\newcommand{\critp}{{\mathcal Critp}} 
\newcommand{\critv}{{\mathcal Critv}} 
\newcommand{\singp}{{\mathcal Singp}} 
\newcommand{\singv}{{\mathcal Singv}} 
\newcommand{\rep}{{\mathcal Rep}} 

\usepackage{epsfig,color}

\begin{document}

\title{Bi-filtrations and persistence paths for 2-Morse functions}
\author{Ryan Budney \\ Tomasz Kaczynski}

\address{
Mathematics and Statistics, University of Victoria PO BOX 3060 STN CSC, Victoria BC Canada V8W 3R4 \\
D\'epartement de math\'ematiques, Universit\'e de Sherbrooke, Sherbrooke, QC, Canada J1K 2R1}
\email{rybu@uvic.ca \\ Tomasz.Kaczynski@usherbrooke.ca}

\thanks{
\newline\noindent
Research of T.K.\ was supported by a Discovery Grant from NSERC of Canada.
\newline\noindent\noindent\emph{Primary class:} 57R35
\newline\noindent\emph{secondary class:} 55N31, 55M99
\newline\noindent\emph{keywords:} Filtration, Manifold, Persistent homology}

\begin{abstract}
This paper studies the homotopy-type of bi-filtrations of compact manifolds
induced as the pre-image of filtrations of $\Real^2$ for generic smooth functions $f : M \to \Real^2$.
The primary goal of the paper is to allow for a simple description of the multi-graded persistent homology
associated to such filtrations. The main result of the paper is a description of the evolution of the bi-filtration of $f$
in terms of cellular attachments.
An analogy of Morse-Conley equation and Morse inequalities along so called persistence paths are derived.
A scheme for computing path-wise barcodes is proposed.
\end{abstract}

\maketitle


\section{Introduction}\label{intro}

In the past two decades, the Morse theory of smooth functions on manifolds, and singularity theory, its extension to functions with multi-dimensional values, have driven a lot of attention in the applied mathematics and theoretical computer science communities  due to their applications in Imaging, Visualisation and, most recently, Topological Data Analysis (TDA). As those theories has been extensively developed for nearly one century, the new and potential applications bring different perspectives.

\medskip

Morse theory is a tool that allows one to use real-valued functions on a manifold to give a combinatorial description of that
manifold, in the language of handle decompositions or CW-complexes.  A topological model for $M$ is built following changes in
 sublevel sets $M_{g\leq y}=g^{-1}( (-\infty,y])$ of a Morse function (i.e.\ smooth and generic) $g:M\to \Real$. The central
  theorem \cite{Mil} about the filtration of $M$ by sublevel sets is that:
\begin{enumerate}
\item[(1)] The homotopy-type of $M_{g\leq y}$ does not change for $y \in [a,b]$ provided there are no critical values of $g$ in the interval $[a,b]$.
\item[(2)] If there is precisely one critical value in $[a,b]$ then the $M_{g\leq b}$ is obtained from $M_{g\leq a}$ by a handle attachment, which up to a homotopy-equivalence, is a cell attachment.
\end{enumerate}

\medskip

In Imaging, TDA, and related techniques of persistent homology, the interest shifts to the function itself. Its domain comes from interpolation of data and is most often a fairly well understood space such as $R^n$ or, if we seek compactness, a triangulated sphere $S^n$. That is a typical setting in works on the shape similarity by size function methods such as in \cite{BiCeFrGiLa}. When it comes to the study of functions with multi-dimensional values, there are new challenges and more differences between the classical singularity theory setting and the applied context.

\medskip

Given the success of Morse theory, the study of generic smooth mappings from manifolds to surfaces $f : M \to \Sigma$
is a natural next-step.
The most basic elements of the theory involves the description of the stratification of the manifold $M$ by the singularity
types, together with the local properties of the mapping around singular points.  This was first worked out by Whitney
\cite{Whit} (see also \cite{GP}) when $M$ is a surface, and fully generalized in the subsequent decades \cite{Wan, Saeki}.
Perhaps the main difference between the study of functions taking values in $\Real$ vs. in a surface is that the
set of fibers $\{ f^{-1}(a) \mid  a\in \Sigma\}$ lack a linear order on them: a poset relation has to be taken into account.
In contrast, the real numbers have the relatively canonical poset $\{ (\infty,a] \mid a \in \Real\}$ of half-infinite intervals.
\medskip

To be specific, let us state the posets studied in this paper. Consider $f: M \to \Real^2$, where $M$ is an
$m$-manifold of dimension $m \geq 2$, and the plane $\Real^2$ is endowed with the poset relation
\[
(a,b)\preceq (a',b')\iff a \leq a' \mbox{ and } b \leq b'.
\]
Any such function gives rise to a {\em bi-filtration}
of $M$ defined as the family $M_f=\{M_{(a,b)}\}_{(a,b)\in \Real^2}$ of subsets of $M$ given by
\[
M_{(a,b)}=\{p\in M\mid f(p)\preceq (a,b)\}.
\]
Equivalently, the sets $M_{(a,b)}$ are the preimages of the quadrants $C_{(a,b)}=(-\infty,a] \times (-\infty,b]$ under $f$.
They are nested with respect to inclusions, that is: $M_{(a,b)}\subseteq M_{(a',b')}$, for every $(a,b)\preceq (a',b')$.

\medskip

Persistence consists of analyzing homological changes occurring along the bi-filtration as the point $(a,b)$ varies. Note that the boundary $\partial C_{(a,b)}$ of the quadrant $C_{(a,b)}$ is not a submanifold of $\Real^2$: it can be viewed as a manifold with a corner.
The problem of bi-filtration has been addressed in the presented setting by Smale in 1975 \cite{Smale} and further investigated by Wan \cite{Wan}. As it is pointed out by Smale, the study is historically motivated by the {\em Pareto optimal problem} of simultaneously maximizing several functions. Our work is an extension of the work done in \cite{Smale,Wan}, with the same topic viewed from a different perspective.

\medskip

There has been a lot of progress in computing persistent homology for multi-filtrations which include function $g:M\to R^k$ as a special case, for any $1< k < \dim M$. We refer the reader to \cite{CZ09,CDFFC,CEFKL} and references therein. However, most of dimension-independent work on computing persistent homology, often in a discrete setting, is in a sense `geometry-blind', that is not giving much insight to the particular types of singularities one may encounter. Providing that insight is the main motivation for this paper. In particular, in \cite{AlKaLaMa19}, a Forman-type discrete analogy of multidimensional Morse functions is investigated. In the conclusion of that paper, it is pointed that an appropriate application-driven extension of the Morse theory to multifiltrations for smooth functions is not much investigated yet, and it would help in understanding the discrete analogy. The present work is a step in that direction. Moreover, Section~\ref{sec:per-paths} shows prospects for practical implementations. A study of discrete Forman type multidimensional Morse functions is currently under way by Landi at.\ al.\ for instance in \cite{Lan-Sca21}.

\medskip

This paper follows the above outline.  We begin Section \ref{sec:fold-cub-sing} with the definition of a $2$-Morse
function $f : M \to \Real^2$, following \cite{Wan, GK-indefinite}.  This allows us to define the (oriented)
signature invariant (see Definition \ref{2morsdef}).  We follow this definition with a few simple examples where
one can explicitly compute the homotopy-types of the filtration $M_{(a,b)}$ for all $(a,b) \in \Real^2$.
The main result of the paper follows.
\begin{theorem}\label{mainth-intro}
If $f : M \to \Real^2$ is a $2$-Morse function then the bi-filtration $M_f$
has singular points consisting entirely of corner and tail singular points.
\end{theorem}

In short, this theorem gives us a stratification of the plane $\Real^2$ such that the homotopy-type of
$M_{(a,b)}$ is constant in the connected co-dimension zero strata.  We follow that up by a description of
how the homotopy-type of $M_{(a,b)}$ changes as $(a,b)$ crosses a co-dimension one stratum.

Our Lemma \ref{mainlem} is the analogue of (1), in that it tells us that generically the homotopy-type
of the manifolds $M_{(a,b)} = f^{-1}(C_{(a,b)})$ is locally-constant.  The nature of the proof of Lemma \ref{mainlem} is
significant to the rest of the paper, describing a rather flexible technique of vector field flows, allowing
us to construct conjugate isotopies (i.e. fiber-preserving isotopies, also known as isotopies that are horizontal
diffeomorphisms with respect to the map $f$) in both $M$ and the plane $\Real^2$.  This allows us the freedom to
frame our remaining arguments in the language of how the homotopy-type of $f^{-1}(C_t)$ changes when
$C_t$ is an arbitrary 'smoothly-varying' $2$-manifold in the plane.
There are however some points in the plane where the homotopy-type of $M_{(a,b)}$ does change, this is described in Theorem \ref{mainthm}.  The main feature of Theorem \ref{mainthm}
is that the homotopy-type of $M_{(a,b)}$ changes via handle (or cell) attachments.  In the proof we see one of the
handle attachments comes directly from a classical Morse theory argument.  The second type of handle attachment uses global
features of the singular point set of $f$, and is perhaps best thought of as a Bott-style handle attachment.  We give a
brief account of Bott's variant of Morse theory. Proposition \ref{prop:char-crit-values} summarizes
elements of the proof of Theorem \ref{mainthm}, describing the dimension of the cell attachments in terms of the
oriented signature invariant.  One last feature of Section \ref{sec:fold-cub-sing} is
the observation that at 'cubic' points $(a,b) \in \Real^2$, while the homotopy-type of $M_{(a,b)}$ generally does not change,
the fibrewise homotopy-type (with respect to the map $f$) does change.  Roughly speaking these cubic points correspond to
pairs of cancelling handles (or cells).

\medskip

In Section \ref{sec:per-paths}, we turn our attention to the consequences of our study for the bi-filtered persistent homology.
We briefly review the known its goals and a technique of reducing its computation to one-dimensional persistence. In particular, the foliation method that has been introduced in \cite{BiCeFrGiLa} for size functions, used in \cite{CaDiFe,CDFFC} for multidimensional persistent homology and later named as fibered barcode in the context of persistence modules \cite{Les-Wri15}.
As it is visible in examples of Section \ref{sec:fold-cub-sing}, in the presence of the poset relation, there are multiple ways of building the topology of $M$ by crossing different arcs of the critical value set while respecting the poset relation.
That leads us the Definition~\ref{def:per-path} of persistence paths. It is a new concept which is somewhat analogous to the mentioned foliation method of \cite{CDFFC}.
It also can be viewed as an analogy of a flow induced by the generalized gradient in \cite{Wan}
with this that our persistence paths apply to functions $f$ which may have cycles in the sense of \cite[Definition 6.4 and 6.5]{Wan}.
We prove an analogy of the Morse-Conley equation \cite{RybZeh85} (Theorem~\ref{th:morese-conley-eq}), and derive from it Corollary~\ref{cor:strong-morse} on strong Morse inequalities for persistence paths.
This gives us a flexible family of Morse inequalities associated to $f$, extending the work of Wan \cite{Wan}.
We conclude Section \ref{sec:per-paths} by introducing path-wise barcodes in Definition~\ref{defn:path-barcode} and describing a scheme for computing the barcodes based on a small representable subfamily $\rep(f)$ of all persistence paths.
While Carlsson and Zomorodian outline an argument that there is no complete and discrete invariant of multi-graded
persistent homology \cite{CZ09}, the primary result of this paper strikes a more optimistic note in the case of multi-filtrations
induced by smooth functions.

\medskip

In Section \ref{sec:ext} we discuss some possible future research directions.

\medskip

The authors would like to thank Hyam Rubinstein, Marian Mrozek, and Michael Lesnick for helpful suggestions.

\medskip


\section{Fold and cubical singularities}
\label{sec:fold-cub-sing}

In the classical Morse theory of smooth real-valued functions and, respectively, singularity theory of functions with values either in a $2$-dimensional manifold $\Sigma$, a critical point or singular point is a point $p\in M$ at which $Df(p)$ is not of maximal rank. The corresponding point $c=f(p)$ in the target space is called a critical, respectively, singular value of $f$. The terminology found in the literature is not consistent: sometimes the terms critical and singular are interchanged.

\medskip

In computational topology, we deal with non-smoothness and degeneracy, so a topological definition is more appropriate. It is also helpful in describing handle attachments. In addition, in the presence of the poset structure of bi-filtrations, as we shall see soon, there is a substantial difference between singularity in the differential sense and criticality in the topological sense. We shall adopt the following definition.

\begin{definition} \label{def:homot-reg-crit-value}
A {\em homotopy regular value} of $f$ {\em with respect to the bi-filtration} of M is a point
$(a,b)\in \Real^2$ such that in some neighbourhood $U_{(a,b)}$ of $(a,b)$,
for all $(a',b'), (a'',b'')\in U_{(a,b)}$ with
\[
(a',b')\preceq (a'',b''),
\]
the inclusion $M_{(a',b')} \hookrightarrow M_{(a'',b'')}$
is a homotopy equivalence. If this condition fails, $(a,b)$ is called a {\em homotopy critical value}.
\end{definition}

A weakening of this definition suited to persistent homology is the notion of {\em homological regular} and {\em critical values} defined by replacing homotopy equivalence by isomorphism induced in homology. This coincides with the definition of  given in \cite[Definition 3.4]{CEFKL}.

\medskip

When $f:M\to \Real$ is a Morse function, the sets of critical points and values in the differential and topological sense coincide. For $\Real^2$-valued functions, even the generic ones, they are substantially different. We shall adopt the terms of {\em singular} points and values for those given by differential definition and {\em critical} to those given by Definition~\ref{def:homot-reg-crit-value}. Given $f:M\to \Real^2$. we consider the following sets:
\[
\singp = \{ p\in M \mid \rank Df(p)< 2\},~~ \singv = f(\singp),
\]
\[
\critv = \{ (a,b) \in \Real^2 \mid \mbox{$(a,b)$ is homotopy critical} \},~~ \critp = f^{-1}(\critp).
\]
We shall soon see that the arcs of $\singv$ along which both coordinates $(a,b)$ increase are homotopy regular, so they are not subsets of $\critv$. The topologically significant arcs are those whose normal vectors have both coordinates of the same sign. Conversely, $\critv$ contains horizontal or vertical half-lines passing through the vertex of $C_{(a,b)}$ and `kissing' points on the singularity $\singv$ but not contained in it. In Proposition~\ref{prop:char-crit-values} we give a classification of different types of criticality.

As we just noticed, Definition~\ref{def:homot-reg-crit-value} also applies to points $(a,b)\in \Real^2$ which are not necessarily the values of $f$, that is, are not in the image $f(M)$. For that reason we will refer to them as to points rather than values and whether we speak about points in $M$ or in $\Real^2$ should be made clear from the context.

We begin with a definition from Gay and Kirby \cite{GK-trisecting, GK-indefinite}, and earlier Wan \cite{Wan}.

\medskip
%
\begin{definition}\label{2morsdef}
A {\it $2$-Morse function} (also called Morse $2$-function)  is a smooth function
$$ f : M \to \Sigma$$
where $M$ is an $m$-manifold, $m \geq 2$ and $\Sigma$ is a $2$-manifold satisfying a local
condition.
For any point $p \in M$ there are neighbourhoods
$U_p \subset M$ of $p$ in $M$ and $V_{f(p)} \subset \Sigma$ of $f(p)$ in $\Sigma$ together
with diffeomorphisms $\phi : U_p \to U'_0 \subset \Real^m$ and $\psi : V_{f(p)} \to V'_{0}$.
with $\phi(p)=0, \psi(f(p))=0$ making the diagram below commute
$$\xymatrix{ U_p \ar[r]^{f_{|U_p}} \ar[d]^{\phi} & V_{f(p)} \ar[d]^{\psi} \\
 U'_0 \ar[r] & V'_{0}}$$
where the bottom horizontal arrow must be one of the three:
\begin{itemize}
\item $(x_1,x_2,\cdots,x_m) \longmapsto (x_1,x_2)$ for this, $p$ is a {\bf regular point}.
\item $(x_1,x_2,\cdots, x_m) \longmapsto (x_1, \pm x_2^2 + \cdots + \pm x_m^2)$
for this, $p$ is a {\bf fold point}.
\item $(x_1,x_2,\cdots,x_m) \longmapsto (x_1, x_2^3 + x_1x_2 + \pm x_3^2 + \cdots + \pm x_m^2)$
for this $p$ is a {\bf cubic point}.
\end{itemize}
\end{definition}

Just like with Morse functions there are elementary transversality conditions equivalent to Definition
\ref{2morsdef} \cite{Wan} \S 1.  This allows the conclusion that for any smooth function
$f : M \to \Sigma$ where $\Sigma$ is a $2$-manifold, via a small perturbation of $f$ we
may convert $f$ into a $2$-Morse function, i.e. $2$-Morse functions form an open and dense subset of
the space of smooth functions $M \to \Sigma$.

The curves of the fold singularities come equipped with transverse-oriented
indices.  This is analogous to the index of a critical point of a Morse function, but made slightly more complex by the
codomain of our function being $\mathbb R^2$.

The index has the form of a triple $(v,i,j)$ where $v$ is a vector transverse to the singular value set, and $i$ is the
dimension of the eigenspace that is folded into the $v$ direction, while $j$ is the dimension of the eigenspace that is
folded into the $-v$ direction.  Thus $i+j=m-1$.   Due to this convention we need the equivalence relation
$(v,i,j) \sim (-v, j, i)$.   Further notice that due to the nature of the cubic
singularity there are two fold-type singularities that merge, with one fold being of type $(v,i,j)$ and
the other fold being of the type $(v,i+1,j-1)$. In our diagrams we will typically draw the $v$ vectors, and only
plot the pair $(i,j)$. In general $i$ is an integer in the set $\{0,1,2,\cdots,m-1\}$, see Figure \ref{cvssym}.
Our oriented index makes sense only on the fold points.
We give a more precise definition in the next paragraph.

\begin{figure}
{
\psfrag{im}[tl][tl][0.8][0]{$(i,j)$}
\psfrag{mi}[tl][tl][0.8][0]{$(j,i)$}
\centerline{\includegraphics[width=8cm]{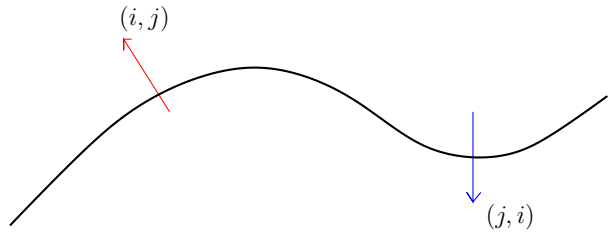}}
}
\vskip 1mm
\caption{Depiction of the symmetry of the index of fold points.}\label{cvssym}
\end{figure}

\begin{definition}\label{idxdef}
Given a Morse $2$-function $f : M \to \Real^2$, and a point $p \in M$ in the fold singular points, then
$$Hf_p : T_p M \otimes T_p M \to \Real^2$$
is a bilinear function taking values in a $1$-dimensional subspace of $\Real^2$, complimentary to the image
of $Df_p$.  Choosing $v \in \Real^2$ spanning this subspace, we can treat $Hf_p$ as real-valued bilinear function, i.e.
by considering $Hf_p \cdot v : T_p M \otimes T_p M \to \Real$.  As this is a symmetric bilinear function, Sylvester's
law of inertia gives us a well-defined signature invariant, $(i,j)$, that can be thought of as the dimensions of the
maximal subspaces where the form is positive or negative-definite respectively.
\end{definition}

Notice that at a cubic singular point the Hessian is degenerate, i.e. $i+j=m-2 < m-1$, with the null space
together with the image of $Df_p$ spanning the cusp's plane of curvature.

Before we begin the examples, it is important to be aware that a Morse function $f : M \to \Real$ gives rise
to a cell decomposition of $M$ \cite{Mil}. These cell decompositions are computable in terms of flow lines of vector
fields conditioned by the derivative of $f$.  The cellular descriptions of $M$ in their most natural state are
homotopy-theoretic in nature, i.e. these techniques give homotopy-equivalences between $M$ and CW-complexes, not
homeomorphisms.  That said, CW-complexes are far from ideal tools to describe manifolds.  The adaptation of CW-complexes
to smooth manifolds are called {\it handle decompositions}, developed by Smale in his proof of the
h-cobordism theorem.  A $k$-cell for an $m$-manifold $M$ is a map $D^k \to M$ that satisfies various properties,
such as being an embedding on the interior. A $k$-handle for an $m$-manifold is a smooth embedding $D^k \times D^{m-k} \to M$,
i.e. handles are not only fully-embedded, but they contain the data of both the cell and a tubular neighbourhood of the
cell.  This allows handle decompositions to not just describe the homotopy-type of $M$, but also its smooth structure.
A sublety of handle attachments is that a $k$-handle is attached only on part of its boundary, i.e.
$(\partial D^k) \times D^{m-k}$, thus there is a risk that we are entering the class of manifolds with cubical corners.
The exposition of Kosinski \cite{Kos} gives careful consideration to this problem, keeping the constructions purely in the
language of manifolds with boundary.  A Morse function $f : M \to \Real$ gives a handle decomposition of $M$, moreover this
handle decomposition describes the smooth structure on $M$.

Starting from illustrative examples, we investigate the relation between bi-filtration and the classical singularity theory.

\begin{figure}
{
\psfrag{r1}[tl][tl][0.8][0]{$\textcolor{green}{r_1}$}
\psfrag{rk}[tl][tl][0.8][0]{$\textcolor{green}{r_k}$}
\psfrag{I1}[tl][tl][0.8][0]{$(m-I_1, I_1)$}
\psfrag{I2}[tl][tl][0.8][0]{$(m-I_k, I_k)$}
\psfrag{i1}[tl][tl][0.8][0]{$H^{I_1}$}
\psfrag{i1m}[tl][tl][0.8][0]{$H^{I_1+1}$}
\psfrag{mi1}[tl][tl][0.8][0]{$H^{m-I_1}$}
\psfrag{i2}[tl][tl][0.8][0]{$H^{I_k}$}
\psfrag{i2m}[tl][tl][0.8][0]{$H^{I_k+1}$}
\psfrag{mi2}[tl][tl][0.8][0]{$H^{m-I_k}$}
$$\includegraphics[width=14cm]{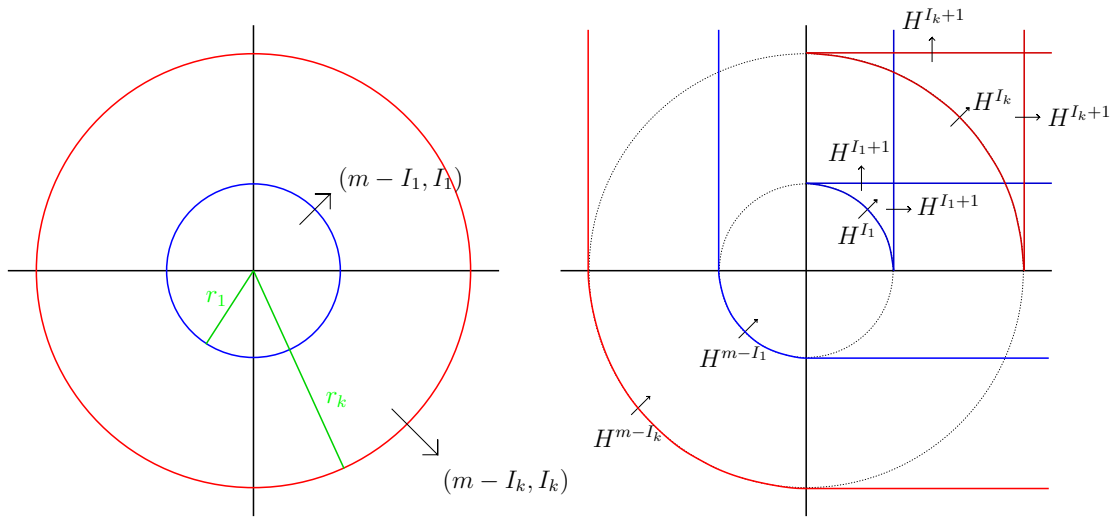}$$
}
\caption{Singular values for $f$ in Example \ref{1-morse} (left) and persistence diagram (right)}\label{rot2mors}
\end{figure}
\vskip 1cm

\begin{example}\label{1-morse}
If $g : M \to \Real$ is a Morse function, then

$$f : S^1 \times M \to \Real^2 $$

given by $f(z,p) = z \cdot g(p)$ is a $2$-Morse function with only fold singularity types.

If the singular values of $g$ consist of positive real numbers $0 < c_1 < c_2 < \cdots < c_k$
then the singular values of $f$ consist of the circles of radius $c_1, c_2, \cdots, c_k$ centred
at the origin.

If the singular value $c_i$ (of $g$) has index $I_i$, then the circle of
$f$ at radius $c_i$ is also a fold-type singular value set of index $(\hat r, m-I_i, I_i)$
where $\hat r$ is the unit outward-pointing radial vector.

The persistence diagram for the preimages $M_{(a,b)}$
is a union of the {\em descending part} of the singular values of $f$ together with some
vertical and horizontal lines at the end-points.

In Figure \ref{rot2mors}, the diagram on the left depicts the singular values of the function $f$.
These are the circles of rotation of the singular
values of $g$.  Say the red circle corresponds to a singular point of index $I_i$.  An alternative way of saying this
is that the homotopy-type of the space $g^{-1}\left( (-\infty,t]\right)$ as $t$ transitions through the point $c_i$ changes by a $I_i$-cell attachment.

In the figure on the right, we describe how the preimages $M_{(a,b)}$ change as the points $(a,b) \in \Real^2$ vary.

Let us take the blue circle for example, on the left.  This is the singular value $c_1$ of index $I_1$.
On the right, this singular circle gives us two singular arcs.  The lower blue arc is properly embedded
in $\Real^2$, and transitioning through it corresponds to a $m-I_1$ handle attachment.  This should be thought
of as a dual handle to $c_1$ (of $g$).  The other singular value is a `fishtail', divided into three properly
embedded arcs.  The round arc corresponds to a handle attachment of index $I_1$, while the two straight lines
correspond to handle attachments of index $I_1+1$ respectively.   The handles of index $I_1+1$ should be thought
of as cancelling handles to the index $I_1$ handle.  Thus attachment of all three handles of index $I_1, I_1+1, I_1+1$
respectively has the same effect on the homotopy-type as a single attachment of a handle of index $I_1+1$.

\end{example}

\begin{example}\label{spherepocket}
Given a round sphere $S^2 \subset \Real^3$, the orthogonal projection map $\pi : \Real^3 \to \Real^2$
when restricted to $S^2$ has singular values the unit circle, corresponding to an equatorial circle in $S^2$ of
singular points.  Imagine the sphere being made of rubber.  We grab a small section of the sphere (away from the
equator) and fold it over itself, creating a cupped sphere.  This introduces an eye singularity in the projection map,
as depicted below.

\begin{figure}
\psfrag{c10}[tl][tl][0.5][0]{$(1,0)$}
\psfrag{c01}[tl][tl][0.5][0]{$(0,1)$}
\psfrag{h0}[tl][tl][0.5][0]{$H^0$}
\psfrag{h1}[tl][tl][0.5][0]{$H^1$}
\psfrag{h2}[tl][tl][0.5][0]{$H^2$}
\psfrag{p0}[tl][tl][0.5][0]{$0$}
\psfrag{p1}[tl][tl][0.5][0]{$1$}
\psfrag{p2}[tl][tl][0.5][0]{$2$}
\psfrag{p1t}[tl][tl][0.5][0]{$1+t$}
\psfrag{p1t2}[tl][tl][0.5][0]{$1+t^2$}
$$\includegraphics[width=12cm]{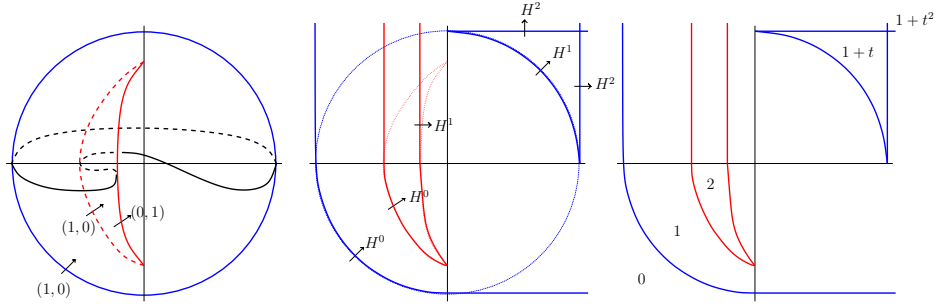}$$
\caption{Cupped sphere projection}\label{cup2sph}
\end{figure}

The pure fold singularity, the equator, is in blue. The red singularity
is an 'eye' type singularity, with precisely two cubic (cusp) points. This is
depicted on the left.  In the central figure we describe the handle attachments of
the bi-filtration.  In the figure on the right we describe the Poincar\'e polynomials
of the bi-filtration, i.e. the bi-filtration is regular at the white points, with transitions
only at the red and blue points.
\end{example}

We give a fairly general example with cubic singularities.

\begin{example}\label{parm-morse}
Cerf theory tells us that a $1$-parameter family of real-valued functions on a manifold is not (generically) Morse at all
parameter times.  There
will be finitely many times where the Morse singularities devolve into cubic singularities.  Thus take a generic $1$-parameter
family of functions on $M$, $F : S^1 \to C(M,\Real)$, and form the function
$$f : S^1 \times M \to \Real^2$$
given by $f(v,p) = F(v)(p)\cdot v$.  The function $f$ is $2$-Morse.  The bi-filtration
$M_{(a,b)}$ will be described in Theorem \ref{mainthm}.
\end{example}

Notice Example \ref{parm-morse} is a direct generalization of Example \ref{1-morse}, i.e.
Example \ref{1-morse} can be derived by setting $F$ to be the constant function.

\begin{example}\label{kleinbot}
A rather colourful example comes from orthogonal projection $\Real^4 \to \Real^2$ pre-composed with
one of the standard embeddings of the Klein bottle $K^2 \to \Real^4$.  This example appears in \cite{Wan}.
The singularity theory for mappings of $2$-manifolds into the plane, of which this is a good demonstration,
 was originally discovered by Whitney \cite{Whit}.
\begin{figure}
\psfrag{i0}[tl][tl][0.7][0]{$(0,1)$}
\psfrag{i1}[tl][tl][0.7][0]{$(1,0)$}
\psfrag{h0}[tl][tl][0.6][0]{$H^0$}
\psfrag{h1}[tl][tl][0.6][0]{$H^1$}
\psfrag{h2}[tl][tl][0.6][0]{$H^2$}
\psfrag{p0}[tl][tl][0.5][0]{$0$}
\psfrag{p1}[tl][tl][0.5][0]{$1$}
\psfrag{p1t}[tl][tl][0.5][0]{$1+t$}
\psfrag{p2t}[tl][tl][0.5][0]{$2+t$}
\psfrag{p12t}[tl][tl][0.5][0]{$1+2t$}
\psfrag{p13t}[tl][tl][0.5][0]{$1+3t$}
\psfrag{p14t}[tl][tl][0.5][0]{$1+4t$}
$$\includegraphics[width=12cm]{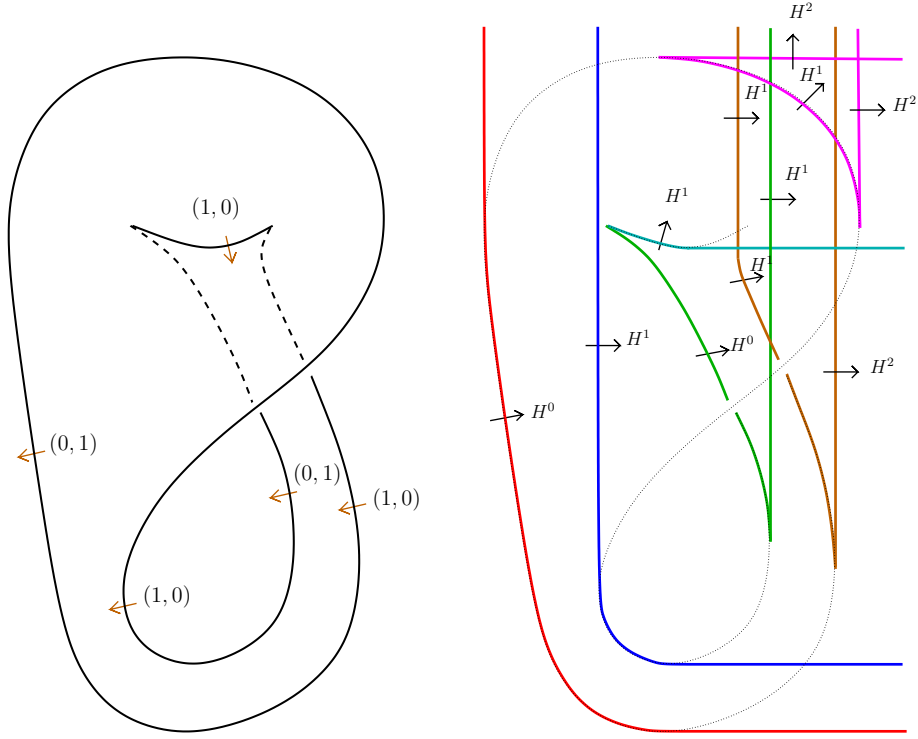}$$
\caption{Klein bottle projection}\label{kleinproj}
\end{figure}
\end{example}

\medskip

We have seen in the previous examples that the singular points of the filtration consist of a subset of the
singular points of the mapping $f$, together with some regular points of the original mapping -- these
were a collection of vertical and horizontal rays.  We divide singular points of the filtration into two
classes {\em corner singular points}, and {\em tail singular points}.

\begin{itemize}
\item We say a point $(a,b) \in \Real^2$ is a corner singular point, if for all suitably-small neighbourhoods
$U$ of $(a,b)$ in $\Real^2$ there are points $(a',b') \in U$ such that $U \cap C_{(a',b')}$ intersects the singular set of
$f$ in both the horizontal and vertical boundary edges of $C_{(a',b')}$. If we write the coordinates of
$\Real^2$ as $(x,y)$ then this happens when locally writing the singular values of $f$ as the graph of a function
$y(x)$, then the function $y(x)$
would be decreasing at $x=a$. A corner singular point is demonstrated in Figure \ref{int1} (G).
\item  For $(a,b)$ to be a tail singular point, we require that $C_{(a,b)}$ intersects the singular values of $f$ tangentially,
on either the interior of the horizontal or vertical boundary curve. In a neighbourhood of the tangential intersection we require
the singular set to be on one side of the cube, i.e. either contained in the cube or in the closure of its exterior.   Thus
it is a 'kissing' tangency. The two tail singular point types are demonstrated in Figure \ref{int1} (B).
\end{itemize}
\begin{figure}
\psfrag{p}[tl][tl][0.5][0]{$(a,b)$}
\psfrag{A}[tl][tl][0.5][0]{$(R)$}
\psfrag{B}[tl][tl][0.5][0]{$(B)$}
\psfrag{C}[tl][tl][0.5][0]{$(G)$}
$$\includegraphics[width=14cm]{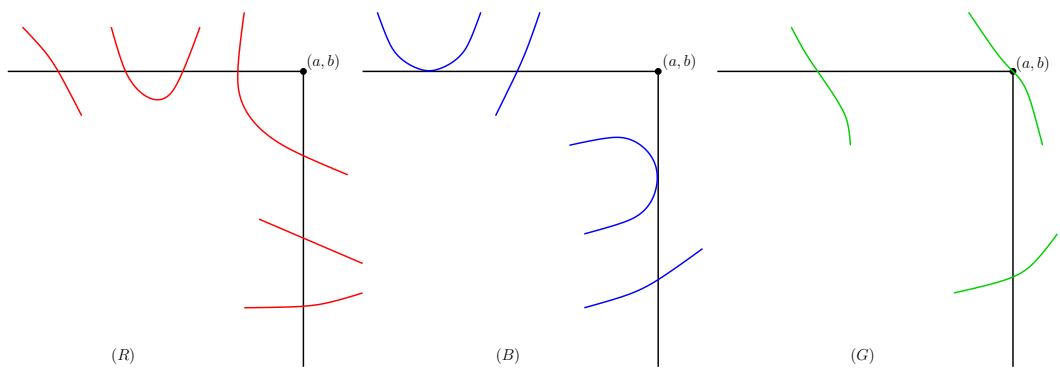}$$
\caption{Three intersection types with dimension $1$ stratum of the singular values.
Type \textcolor{red}{(R)} consists of transverse intersections with the $0$ and $1$-dimensional strata
of the singular value set for $f$.  Type \textcolor{blue}{(B)} consists of both transverse intersections
and simple 'kissing' non-transverse intersections with the $1$-strata of the singular value set.  Type
\textcolor{green}{(G)} consists of transverse intersections together with a corner-type intersection.
Thus (R) is generic, i.e. codimension $0$ in the filtration, while types (B) and (G) are of higher codimension.
}\label{int1}
\end{figure}

Notice that Figure \ref{int1} (R) describes a regular point $(a,b)$ of the filtration $M_f$. We should note that while it is
true a point can be both corner singular and tail singular at the same time, this is a codimension two condition,
thus it is relatively rare.  On the other hand, tail and corner singular points are codimension one conditions.

The proof of Theorem \ref{mainthm} has several special cases, but there is one elemental argument that is common
to all cases.  We put this in the next lemma.

\begin{lemma}\label{mainlem}
If $f : M \to \Real^2$ is a $2$-Morse function, provided the point $(a,b) \in \Real^2$ is regular for $f$,
and the boundary of $C_{(a,b)}$ intersects the singular values of $f$ transversely then $(a,b)$ is not only
a regular point for the filtration $M_f$, but the filtration is {\bf locally trivial} near $(a,b)$.
\end{lemma}

\begin{proof}
We can prove something considerably stronger when $(a,b) \in \Real^2$ is regular -- the filtration is locally
trivial near $(a,b)$.  Precisely, there is a neighbourhood $U$ of $(a,b) \in \Real^2$ such that for any
$(a',b') \in U$ there is a diffeomorphism $\tilde \phi : M \to M$ covering a diffeomorphism $\phi : \Real^2 \to \Real^2$
such that $\tilde \phi (f^{-1}(C_{(a,b)})) = f^{-1}(C_{(a',b')})$ and $\phi(C_{(a,b)}) = C_{(a',b')}$.  When we say
$\tilde \phi$ 'covers' $\phi$ we mean the following diagram commutes.

\begin{figure}
\psfrag{p}[tl][tl][0.5][0]{$(a,b)$}
\psfrag{A}[tl][tl][0.5][0]{$(R)$}
\psfrag{B}[tl][tl][0.5][0]{$(B)$}
\psfrag{C}[tl][tl][0.5][0]{$(G)$}
$$\includegraphics[width=14cm]{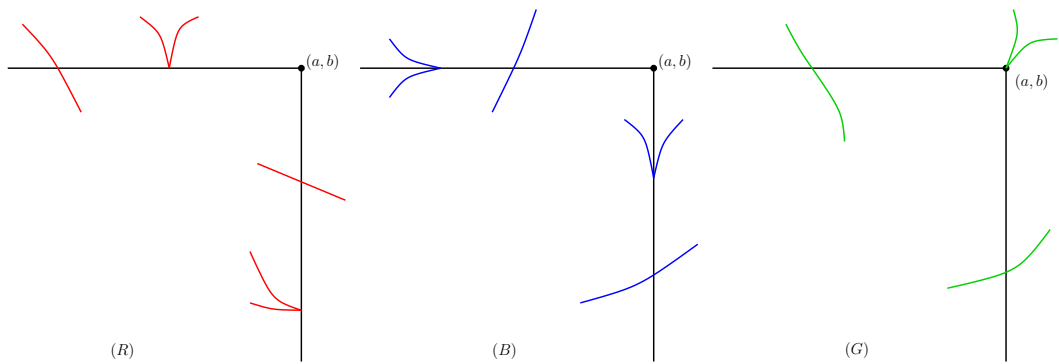}$$
\caption{Three intersection types with dimension $0$ stratum of the singular values.
Type \textcolor{red}{(R)} consists of generic outward and inward intersections with the
$0$-dimensional stratum of the singular value set for $f$.  Type \textcolor{blue}{(B)} consists non-generic
intersections with the $0$-dimensional stratum.  Type
\textcolor{green}{(G)} consists of a generic corner-type intersection with the $0$-dimensional stratum.
In the filtration parameters, \textcolor{green}{(G)} is of codimension two.  Type \textcolor{blue}{(B)} is
of codimension one if it exists, but for a generic $2$-Morse function these singularity types are avoidable,
i.e. one can convert them into type \textcolor{red}{(R)} after a small isotopy applied to $f$.
}\label{int2}
\end{figure}

$$\xymatrix{ M \ar[r]^{\tilde \phi} \ar[d]^-f & M \ar[d]^-f \\
  \Real^2 \ar[r]^\phi & \Real^2 }$$

The map $\tilde \phi$ is sometimes called a fibre-preserving diffeomorphism of $f$, or a {\it horizontal} diffeomorphism.
A consequence of our proof will be that $\tilde \phi$ and $\phi$ are close to the identity diffeomorphism,
where 'close' is controlled by the size of the neighbourhood $U$.  That such a neighbourhood exists can be deduced
from the Transversality Stability Theorem \cite{GP}.

The idea of the proof is to find a vector field in the plane whose flow maps $C_{(a,b)}$ to $C_{(a',b')}$ provided
$(a',b')$ is near-enough to $(a,b)$.  We construct the vector field in a manner that allows us to lift it to
a vector field on $M$, thus the flow of this vector field will send $f^{-1}(C_{(a,b)})$ to $f^{-1}(C_{(a',b')})$.
Given that the derivative of $f$ is not an epi-morphism at singular points of $f$, we have to take some care defining
the vector field.  At fold points of $f$ the derivative of $f$ is only onto the tangent space of the singular-value
set.  Thus our neighbourhood $U$ will be constrained by the sole demand that the singular value set needs to be
transverse to $\partial C_{(a',b')}$ for all $(a',b') \in U$.  An example illustration of a valid $U$ is depicted
in the figure below.  The green region illustrated is the set of points $UC = \{ p \in \partial C_{(a',b')} : (a',b') \in U \}$.

For the sake of argument, let's assume $b'=b$, i.e. we break the proof into two steps, i.e. step 1 with $b'=b$ and
step 2 with $a'=a$.  We further assume $a'>a$ as the $a'<a$ case is analogous. Let $\singv(f)$ denote the singular
values of $f$, i.e. $\singv(f) \subset \Real^2$. Consider the curves of $\singv(f) \cap UC$.  On the path-components
of $\singv(f) \cap UC$ that live in the vertical portion of $UC$, we define the vector field to be the unique vector
field that is tangent to $\singv(f)$, and whose $x$-component unit, in particular positive. On the path components
of $\singv(f) \cap UC$ that are in the horizontal portion of $UC$ we define the vector field to be zero.  In the
horizontal portion of $UC$ we extend the vector field to be zero.  In the vertical portion of $UC$ we interpolate
between the definition on $\singv(f) \cap UC$ and the vector field $(1,0)$, using a tubular neighbourhood of $\singv(f)$
in $UC$. Doing this we can ensure the vector field in the vertical portion of $UC$ always has unit $x$-component.
We extend the vector field to all of $\Real^2$, choosing any extension that keeps the length of the vector field bounded,
i.e. so that its flow is complete.  This gives us a flow on $\Real^2$ that sends $C_{(a,b)}$ to $C_{(a',b)}$.
Our vector field lifts to $M$ since the derivative of $f$ is onto the tangent spaces of $\singv(f)$, and for regular
points, the derivative has rank two. By the existence and uniqueness theorem for solutions to ODEs, the flow of an $f$-lifted
vector field is conjugated (by $f$) to the flow of the original vector field on $\Real^2$.  Thus
the flow on $M$ is fiber-preserving and sends $f^{-1}(C_{(a,b)})$ to $f^{-1}(C_{(a',b)})$.

{
\psfrag{ab}[tl][tl][0.8][0]{$(a,b)$}
\psfrag{apb}[tl][tl][0.8][0]{$(a',b)$}
\psfrag{A}[tl][tl][0.8][0]{$\textcolor{green}{UC}$}
\psfrag{cvf}[tl][tl][0.8][0]{$\textcolor{red}{\singv(f)}$}
$$\includegraphics[width=14cm]{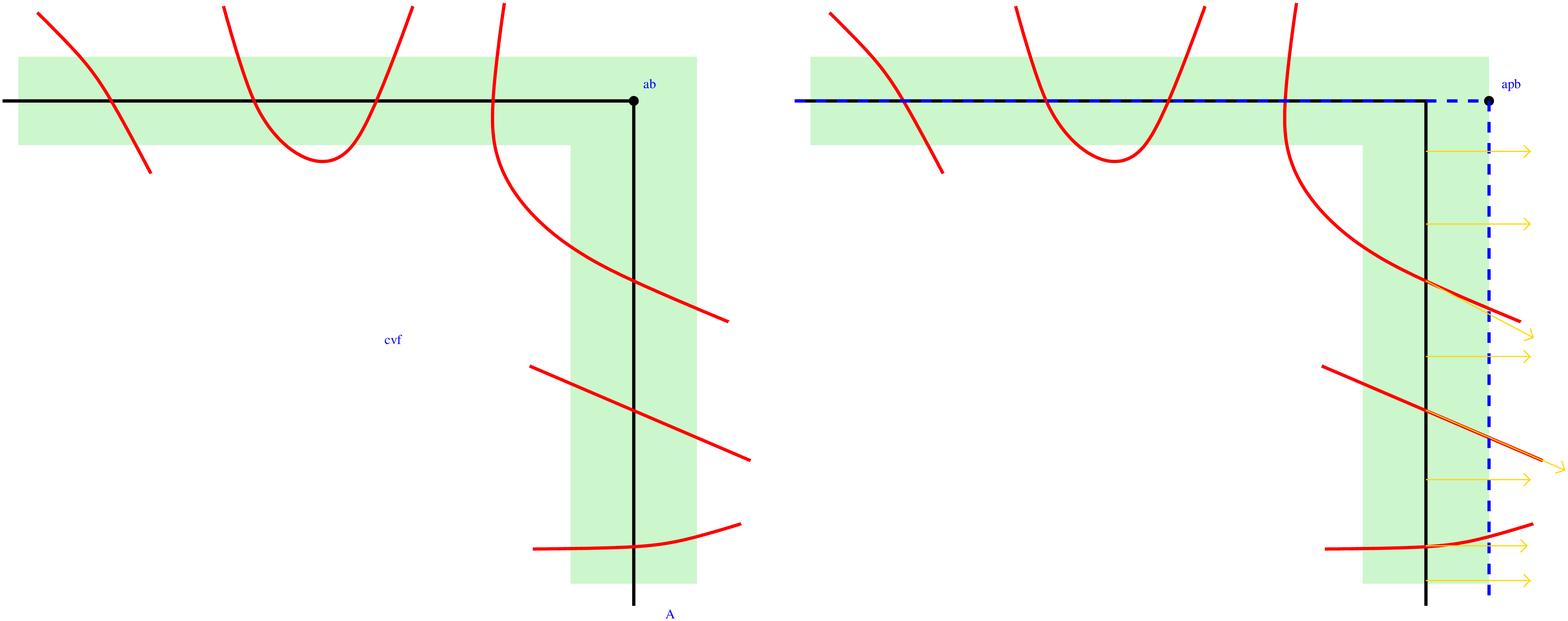}$$
}
\end{proof}

Lemma \ref{mainlem} has several natural generalizations.  For example, let $C$ and $C'$ be compact $2$-dimensional
submanifolds of $\Real^2$.  Then provided there is an isotopy between $C$ and $C'$ such that $\partial C$ is transverse
to the singular value set of $f$ through the entire isotopy (technically one needs to include intersections pairs
of $\singv(f)$ curves as $0$-strata in $\singv(f)$ for this statement to be true),
then $f^{-1}(C)$ and $f^{-1}(C')$ are fiberwise
diffeomorphic.  The condition that the isotopy is transverse to the singular value set through all parameter times
guarantees that the boundary of $C$ does not pass over a cubic point, or ever become tangent to the singular value
set.  These are the events that can trigger changes in the fibrewise homotopy-type of the preimage.

Similarly, provided $(a,b)$ is a regular value of $f$ we can {\it round the corner},
turning $C_{(a,b)}$ into a smooth manifold $C'_{(a,b)}$ such that $f^{-1}(C_{(a,b)})$ and $f^{-1}(C'_{(a,b)})$ are
homotopy-equivalent.


We choose to let $C$ be a compact submanifold of $\Real^2$ for the following arguments, i.e. rather than working with
quadrants $C_{(a,b)}$ we choose to work with compact smooth manifolds, as it exposes the essential features of the argument.

The next theorem states that if the boundary of $C$ (or quadrant $C_{(a,b)}$) passes over a cubic point in the isotopy,
the fiberwise homotopy-type changes but the homotopy-type does not.   Moreover, if the boundary of $C$ passes across the
singular value set (at a tangency or corner, i.e. Figure \ref{int1} (B) and (G)) then the homotopy-type changes via a cell
attachment. We also give enough details that allow the computation of the attaching maps.

\begin{figure}
\psfrag{p}[tl][tl][0.5][0]{$(a,b)$}
\psfrag{A}[tl][tl][0.5][0]{$ $}
$$\includegraphics[width=9cm]{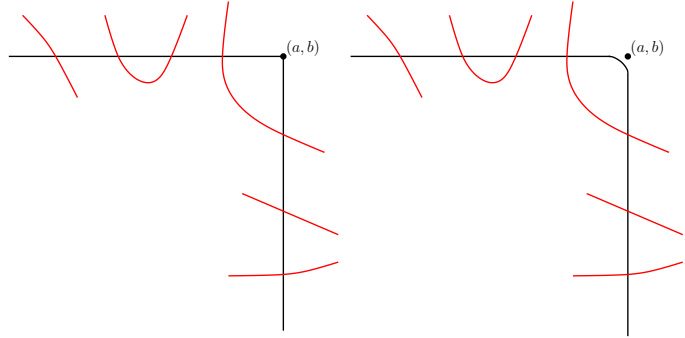}$$
\caption{Rounding $C_{(a,b)}$ to produce $C'_{(a,b)}$.}\label{roundit}
\end{figure}

\medskip

\begin{theorem}\label{mainthm}
If $f : M \to \Real^2$ is a $2$-Morse function then the bi-filtration $M_f$
has singular points consisting entirely of corner and tail singular points.
Further, provided the two height functions $\pi_i : \Real^2 \to \Real$ given by $\pi_1(x,y) = x$ and $\pi_2(x,y)=y$
restrict to Morse function on the fold singular values of $f$, with distinct critical heights
then the transitions to the homotopy-type of $M_{(a,b)}$ when $(a,b)$ is either a corner or tail singular point are given by
individual cell attachments.
\end{theorem}

\begin{proof}
Rather than using the restrictive language of quadrants, let $C$ be a compact submanifold of $\Real^2$ and we investigate
the change in homotopy-type of $f^{-1}(C)$ through an isotopy of $C$.  We have two cases to consider.

Case 1 is a regular tangency -- analogous to a type-$2$ Reidemeister move of the planar diagram, in that it creates
two points of intersection between the boundary of $C$ and the singular value set.  Roughly speaking there are two
types of regular tangency moves. This move can be described via a 'bigon modification' where one appends a bigon to
the manifold $C$, attaching along one of the edges. The second edge of the bigon belongs to $\singv(f)$.  In the
'non-engulfing' move, $\singv(f)$ points out of $C$ after the bigon is appended, while in the engulfing version,
$\singv(f)$ points into $C$ as one departs the bigon.

In the non-engulfing version of case 1, the move corresponds to a cell attachment of index $i$ provided the index
of $\singv(f)$ is of the form $(v,i,j)$ where $v$ is in the direction of $\partial C$ as it sweeps over $\singv(f)$.

Let $C'$ denote the submanifold of $\Real^2$ after the isotopy of $C$ has been applied, i.e. as in the
right hand side of the figure above.  Using an argument analogous to Lemma \ref{mainlem} we see that
$f^{-1}(C')$ has the same homotopy-type as $f^{-1}(C) \cup f^{-1}(B)$ where $B$ is the blue arc below.

{
\psfrag{C}[tl][tl][0.8][0]{$C$}
\psfrag{Cp}[tl][tl][0.8][0]{$C'$}
\psfrag{cvf}[tl][tl][0.8][0]{$\textcolor{red}{\singv(f)}$}
$$\includegraphics[width=12cm]{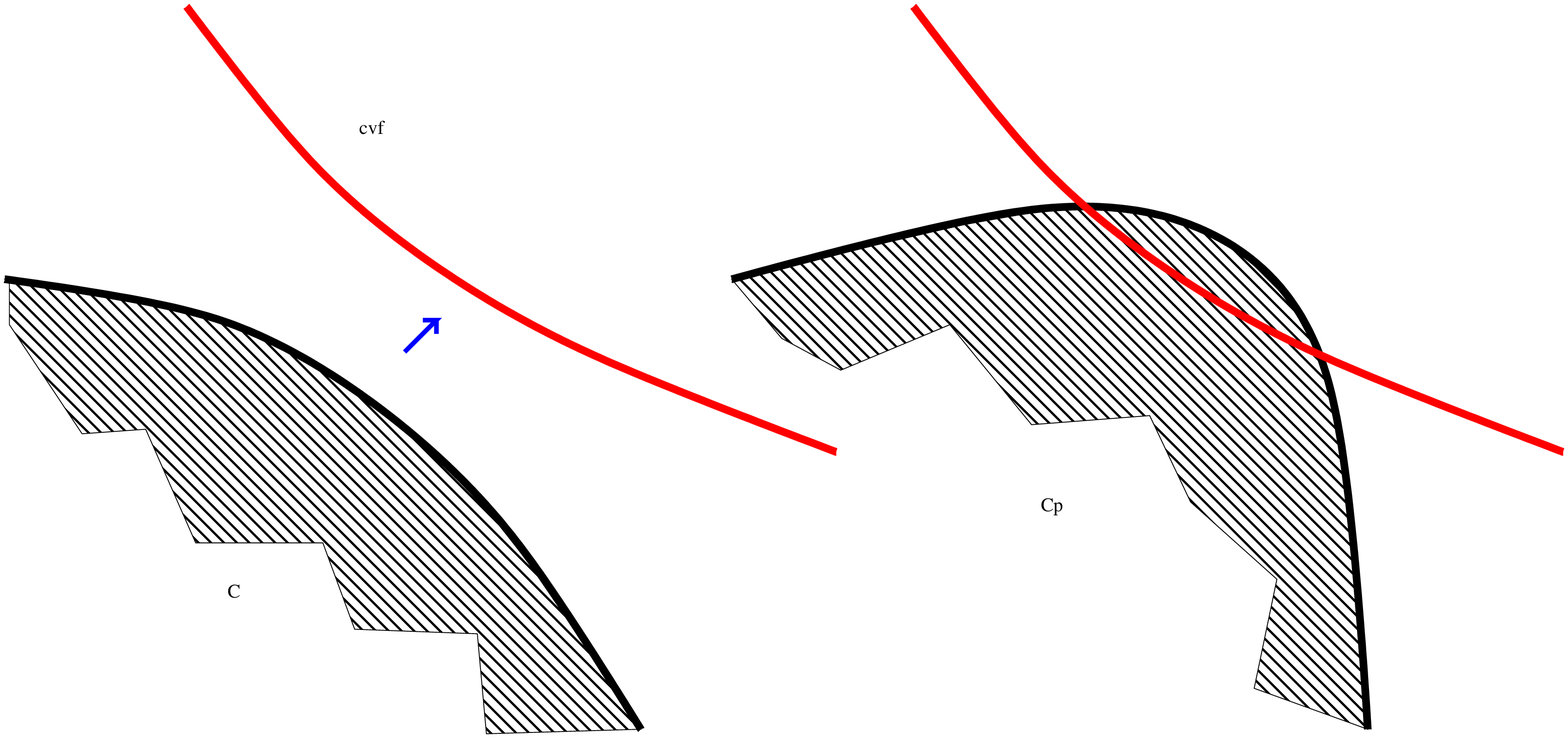}$$
\centerline{Case 1, non-engulfing}
}

The restriction of $f$ to $f^{-1}(B)$, and after identifying $B$ with an interval in $\Real$, is a $1$-Morse function,
thus $f^{-1}(B)$ has the homotopy-type of $f^{-1}(B \cap C)$ attach an $i$-cell, by the Morse Lemma.  More specifically,
this is proven in Theorem 3.2 of \cite{Mil}.

{
\psfrag{C}[tl][tl][0.8][0]{$C$}
\psfrag{B}[tl][tl][0.8][0]{$\textcolor{blue}{B}$}
\psfrag{cvf}[tl][tl][0.8][0]{$\textcolor{red}{\singv(f)}$}
$$\includegraphics[width=5cm]{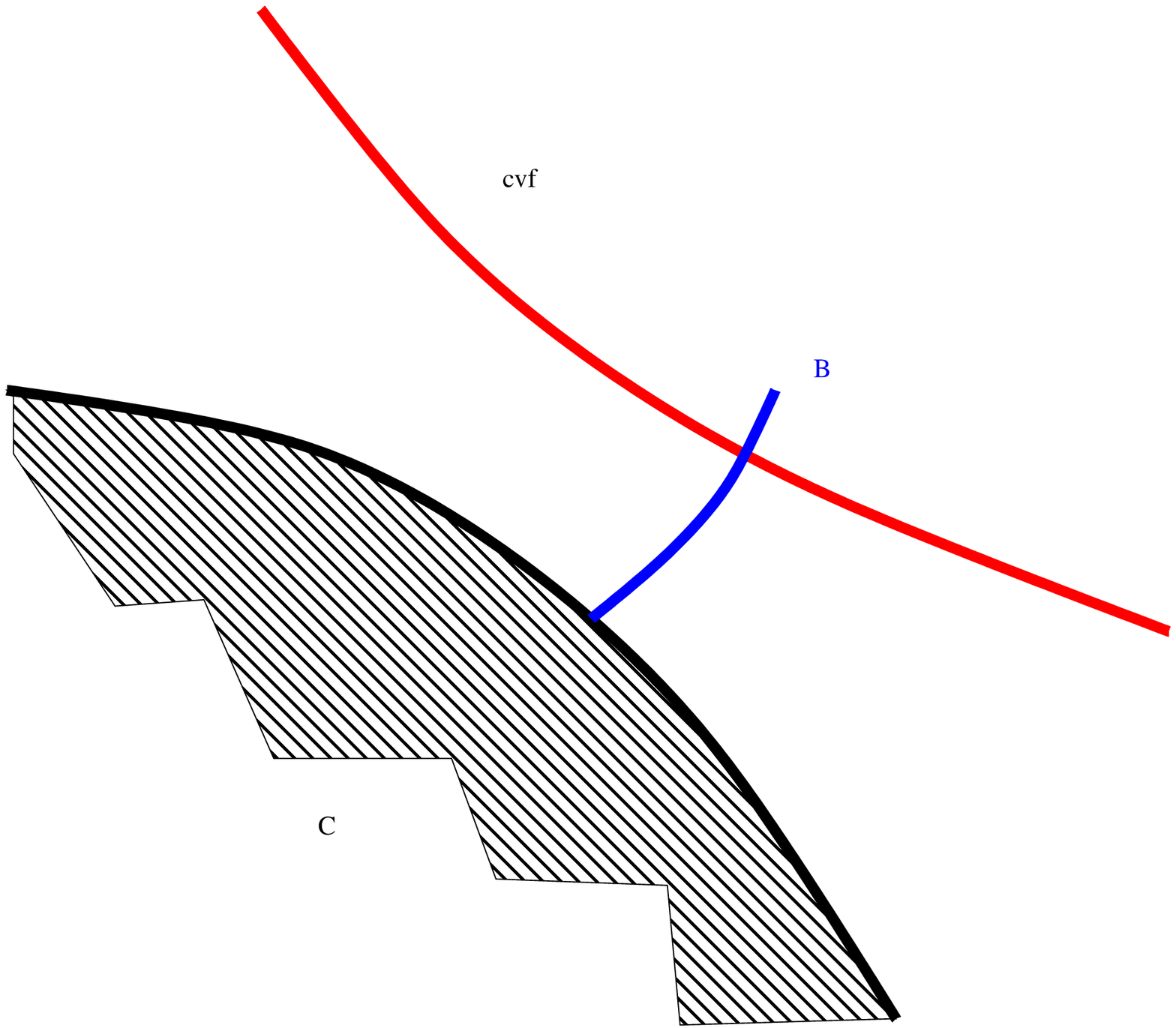}$$
\centerline{Case 1, non-engulfing}
}

For the engulfing version of case 1, the cell attachment is of index $i+1$ provided the index of $\singv(f)$ is $(v,i,j)$,
and the attaching map is analogous to the previous case, but it should be thought of as an unbased version of a Whitehead
product of the attaching map in the previous case, with the red interval disjoint from $C$ in the diagram below.  Specifically,
the characteristic map will be a product $D^i \times I$ where the $D^i$ maps transversely to the red interval, and I can
be identified with the red interval.

{
\psfrag{C}[tl][tl][0.8][0]{$C$}
\psfrag{Cp}[tl][tl][0.8][0]{$C'$}
\psfrag{cvf}[tl][tl][0.8][0]{$\textcolor{red}{\singv(f)}$}
$$\includegraphics[width=10cm]{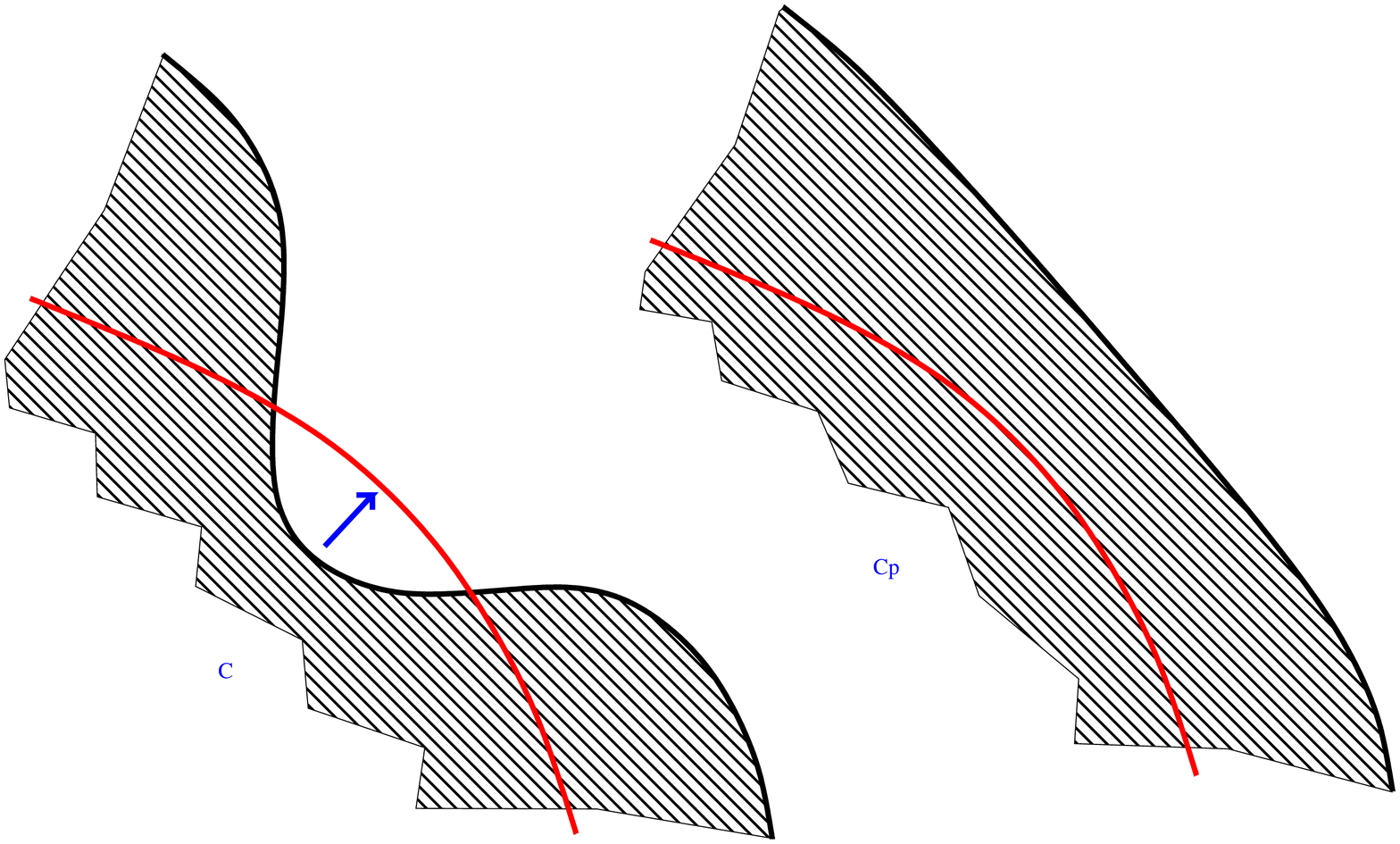}$$
\centerline{Case 1, engulfing}
}

The above figure indicates the rationale.  Specifically, $f^{-1}(C')$ is $f^{-1}(C)$ union a relative Bott-type
handle. This handle should be thought of as $I \times D^i \times D^{m-i-1}$ where $(v,i,j)$ is the index of
$\singv(f)$.  This is because $\pi_v \circ f$ is a Bott-style Morse function on $f^{-1}(B)$ (see the figure below).
The function $\pi_v : \Real^2 \to \Real \cdot v$ is orthogonal projection onto the line spanned by $v$.
The 'box' $B$ is diffeomorphic to a product $B \simeq I \times I$ where the first interval factor corresponds to
the red arc of $\singv(f)$ disjoint from $C$ in the figure below, while the second interval $I$ is in the transverse
direction (i.e. can be taken to be parallel to $v$).  Thus $f^{-1}(B)$ is an interval cross an $i$-handle, being attached
to $f^{-1}(C)$ along $\left(I \times \partial D^i\right) \cup \left((\partial I) \times D^i\right)$, i.e. $\partial(I \times D^i)$.
This could be thought of as an unbased version of a Whitehead product.

For details on Bott-style Morse functions, and how they give disc-bundle adjunctions for manifolds
see the paper of Bott \cite{Bott} (below item 3.6).  For a gentler introduction, see \cite{BH}.

{
\psfrag{C}[tl][tl][0.8][0]{$C$}
\psfrag{B}[tl][tl][0.8][0]{$B$}
\psfrag{cvf}[tl][tl][0.8][0]{$\textcolor{red}{\singv(f)}$}
$$\includegraphics[width=6cm]{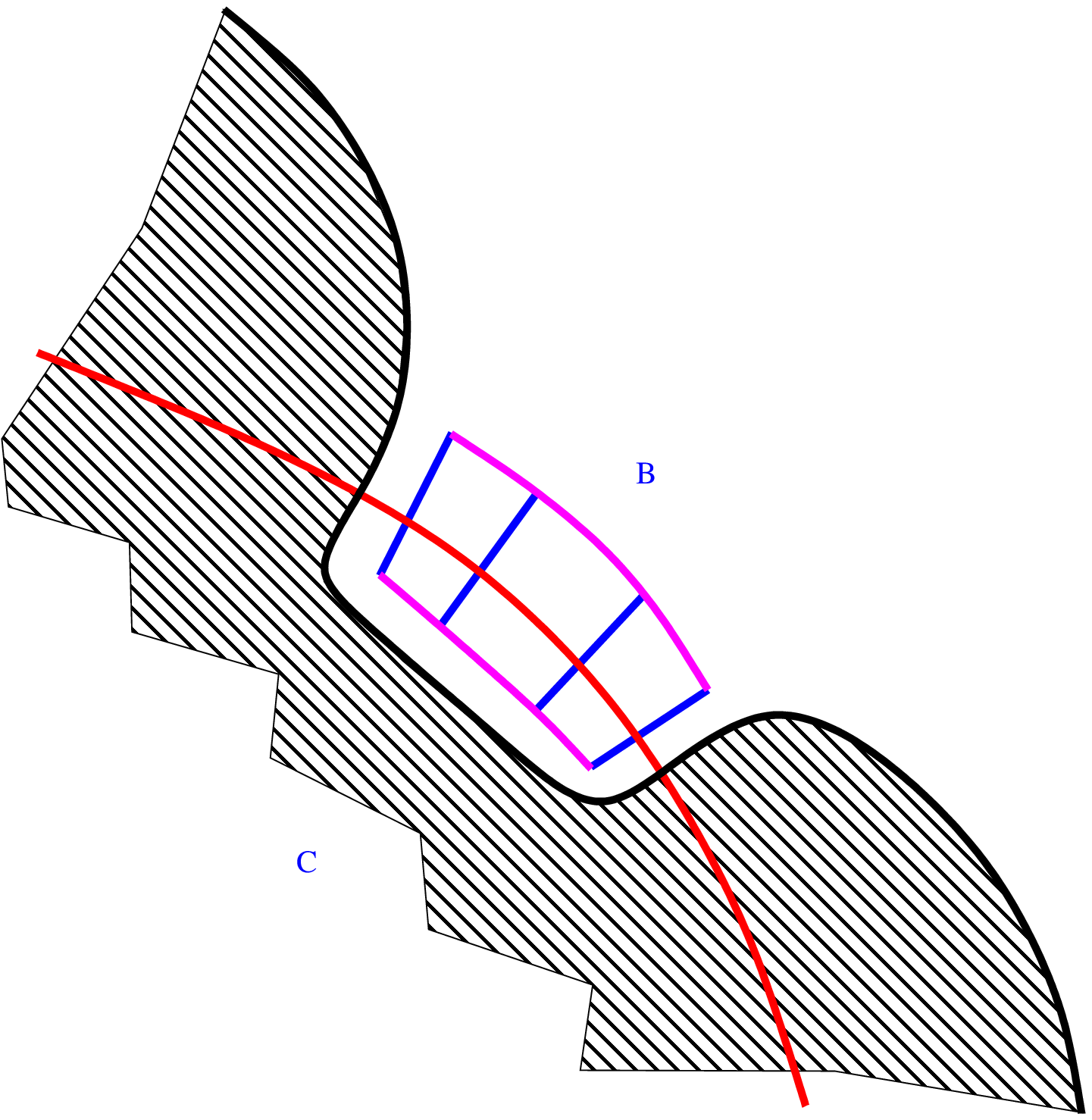}$$
\centerline{Case 1, engulfing}
}

Case 2 is the case where the boundary of $C$ passes over a cubic point.  We will see that the homotopy-type of
$f^{-1}(C)$ does not change in this instance.  Like Case 1 there is are 'engulfing' and 'non-engulfing' sub-cases.
We restrict to the non-engulfing case, as the engulfing case is similar.   The main idea of the proof is that this
transition corresponds to a $1$-parameter family of cancelling $i$ and $(i+1)$-handle attachments, thus we are attaching
a ball along a hemi-sphere, which results in no change in the homotopy-type.

{
\psfrag{C}[tl][tl][0.8][0]{$C$}
\psfrag{CF}[tl][tl][0.8][0]{$\textcolor{red}{\singv(f)}$}
$$\includegraphics[width=12cm]{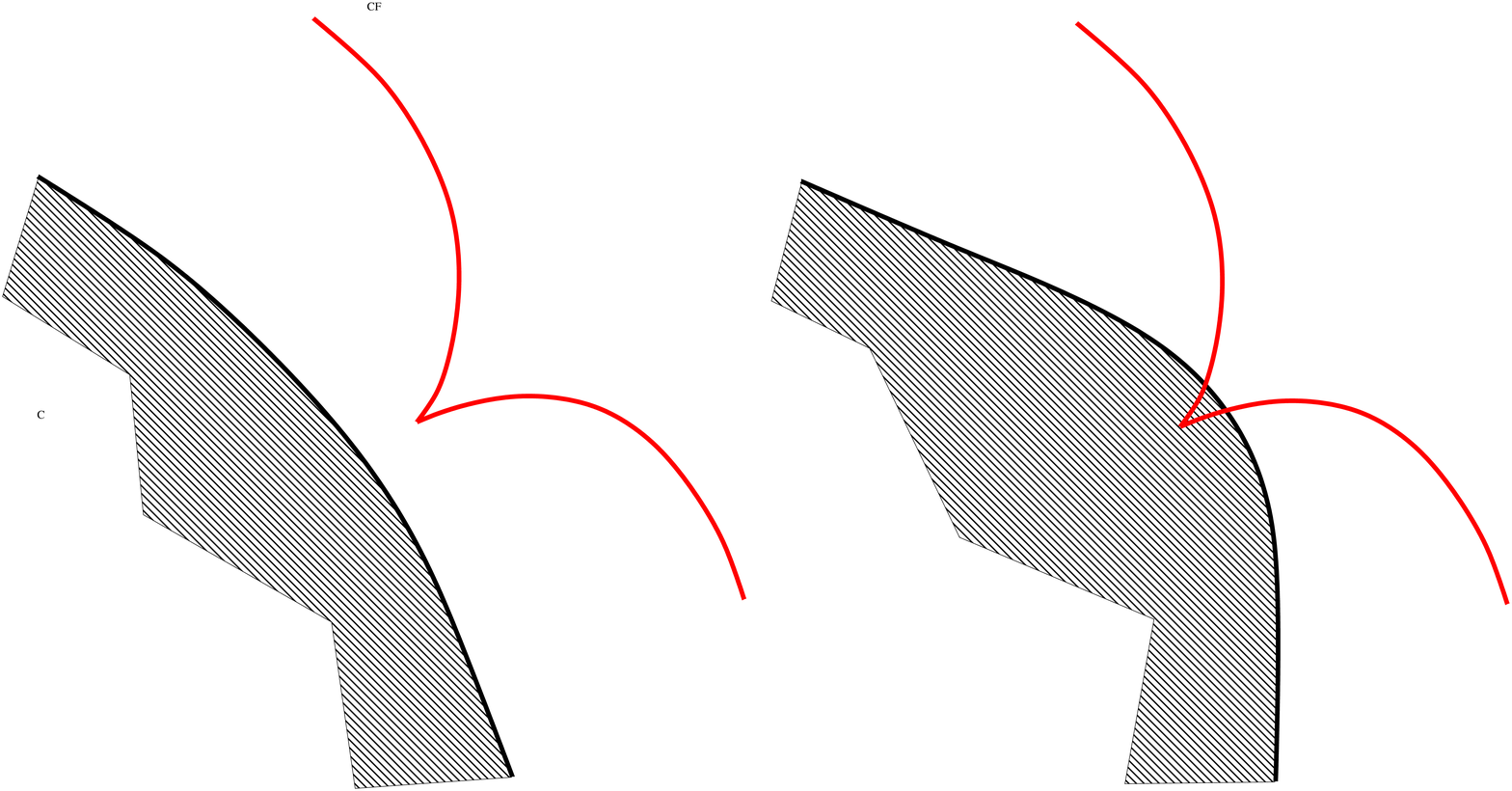}$$
\centerline{Case 2}
}
\end{proof}

We should point out that Case 2 has one special case that does result in a homotopy-type change.  This is depicted in
Figure \ref{int2} (B).  These types of $2$-Morse functions are not generic. A small isotopy of $f$ allows one to ensure the
tangent vectors at the cusps are neither vertical or horizontal, i.e. this at least codimension $1$ in the
space of smooth functions $M \to \Real^2$, thus this situation is avoidable.

Theorem \ref{mainthm} allows us to draw the singular point set of the filtration $M_f$ from the singular values
of $f$, allowing automatic deduction of Examples \ref{1-morse}, \ref{spherepocket}, and Figure~\ref{kleinbot}.

\begin{proposition} \label{prop:char-crit-values}
We summarize the cell attachments at the singular values of the filtration $M_f$.
\begin{itemize}
\item For corner singular points, the index of the cell attachment for $M_f$ is the same as the index for $f$.
\item For tail singular points, if the singular values of $f$ near the kissing tangency are
exterior to the cube, then the cell attachment has the same index as $f$. This corresponds to Wan terminating
a Pareto arc with a positive sign \cite{Wan}.
\item For a tail singular point, if the singular values of $f$ near the
kissing tangency are in the interior of the cube then the index of the cell attachment is one greater than that of
the corresponding singular value for $f$. This correspond's to Wan terminating a Pareto arc with a negative sign \cite{Wan}.
\end{itemize}
\end{proposition}

We should note, Wan \cite{Wan} gives a filtration of the manifold $M$ when $f : M \to \Real^2$ is $2$-Morse, provided the
$2$-Morse function satisfies the {\it no cycle} condition (see Proposition 6.3 \cite{Wan}). Central to Wan's construction
is the usage of flowlines of 'generalized gradient' vector fields -- roughly these are vector fields where both coordinates
are increasing (away from the Pareto points).  When one has a cycle, one can loop endlessly between Pareto points, but
when there are no cycles, the process of connecting Pareto points via paths of generalized gradients exhausts the manifold
$M$ and linearly orders the critical intervals of Pareto sets.
In our work there are a multitude of filtrations whether or not
$f$ has cycles.  All the examples provided so far in this paper (and all examples in Wan's work \cite{Wan}) satisfy
the no cycle condition.

The simplest example of a 2-Morse function with a cycle in Wan's sense is a function of the form
$f : S^1 \times D^2 \to \Real^2$ having two critical arcs of index $(1,1)$, with the critical arcs being
properly embedded in $S^1 \times D^2$.  There are generalized gradient flows on the endpoints
connecting the arcs in a cyclic ordering.  While this function is only defined on a manifold diffeomorphic
to $S^1 \times D^2$, with a little work one can embed this $2$-Morse function into a closed $3$-manifold,
but one needs to add additional critical values.  There is a rather simple cyclic example if one allows
the use of $2$-Morse functions of the sort $f : S^3 \to S^2$.    We obtain this map as the composite of
the $2$-sheeted branched cover $S^3 \to S^3$ over the Hopf link together with the Hopf fibration $S^3 \to S^2$,
provided the Hopf fibration projection of the Hopf link is a $2$-crossing diagram in $S^2$.


\section{Persistence paths and path-wise barcodes}
\label{sec:per-paths}

In a $1$-D persistent homology, barcodes represent collections of parameter intervals at which homology generators are born and killed. In multi-filtered persistent homology, in particular, in our $2$-D case, there is no simple barcode analogy, and, as Carlsson et al. pointed out in \cite{CZ09}, there is no complete discrete invariant. Many authors have studied {\em rank invariants} in a module theory setting \cite{CZ09,Les19}. A somewhat more elementary notion of {\em persistent Betti number (PBN) functions} is presented in Cerri et al. \cite[Definition 2.2]{CDFFC}. These are collections of functions $\{\beta_{f,q}: \Delta^+ \to N \cup\infty\}_{q\in \Zed}$,
\[
\beta_{f,q}((a,b),(a',b'))=\rank H_q(i^{((a,b),(a',b'))})
\]
where
\[
\Delta^+=\{((a,b),(a',b'))\in \Real^2\times \Real^2 \mid (a,b)\preceq(a',b')\},
\]
and
$i^{((a,b),(a',b'))}:M_{(a,b)}\hookrightarrow M_{(a',b')}$ is the inclusion of sublevel sets.

\medskip

For computational purposes, the authors of \cite{CDFFC} use a reduction to one-dimensional persistence diagrams via so called {\em foliation method}. It consists of applying the one-dimensional rank invariant along the lines defined by positive coordinate vectors in chosen finite grids. That method is restated as a {\em fibered barcode} in the context of persistence modules by Lesnick in \cite[Definition 13.6]{Les19}.

As we observed in Section~\ref{sec:fold-cub-sing} on our 2-Morse function examples, although there are uncountably many singular points, the changes in topology can be finitely characterised. They either occur when we cross an arc of the singularity $\singv$ in the poset-increasing direction, or when we cross a horizontal or vertical half-line passing through the vertex $(a,b)$ of $C_{(a,b)}$ and `kissing' the singularity. We will refer to both types of components of $\critv$ as to
{\em Pareto critical value arcs} or, for short, {\em Pareto arcs}. Note that in \cite{Wan}, the term {\em Pareto set} refers to a subset of $\singp\subset M$ and {\em critical intervals} to its components, while our Pareto arcs are the corresponding subsets of $\critv\subset \Real^2$.
There are finitely many homotopically distinct paths, with $M_{(a,b)}$ starting with an empty set and ending with the whole manifold. Each one can give a different sequence of handle attachments creating new generators of homology or cancelling previous ones, all giving $H_*(M)$ at the end of the day.
\medskip

This observation leads to the notion of persistence paths which is a substitute for either Cerri's foliation method \cite{CDFFC} or Lesnick's fibered barcode \cite{Les19}. It can be also be viewed as a discrete analogy of Wan's generalized gradient (whose choice is also not unique) in \cite{Wan}.
Before we proceed, let us introduce some terminology. As far as rank invariants or persistent Betti numbers are of concern, a convenient way to record the homological information carried in sublevel sets is the {\em Poincar\'e polynomial}
\[
P(t,M)=\sum_{k=0}^{n}\beta_k \; t^k,
\]
where $\beta_k=\rank H_k(M)$ and $n$ is the dimension of $M$.

\begin{figure}
\psfrag{i0}[tl][tl][0.7][0]{$(0,1)$}
\psfrag{i1}[tl][tl][0.7][0]{$(1,0)$}
\psfrag{h0}[tl][tl][0.6][0]{$H^0$}
\psfrag{h1}[tl][tl][0.6][0]{$H^1$}
\psfrag{h2}[tl][tl][0.6][0]{$H^2$}
\psfrag{p0}[tl][tl][0.5][0]{$0$}
\psfrag{p1}[tl][tl][0.5][0]{$1$}
\psfrag{p1t}[tl][tl][0.5][0]{$1+t$}
\psfrag{p2t}[tl][tl][0.5][0]{$2+t$}
\psfrag{p12t}[tl][tl][0.5][0]{$1+2t$}
\psfrag{p13t}[tl][tl][0.5][0]{$1+3t$}
\psfrag{p14t}[tl][tl][0.5][0]{$1+4t$}
$$\includegraphics[width=14cm]{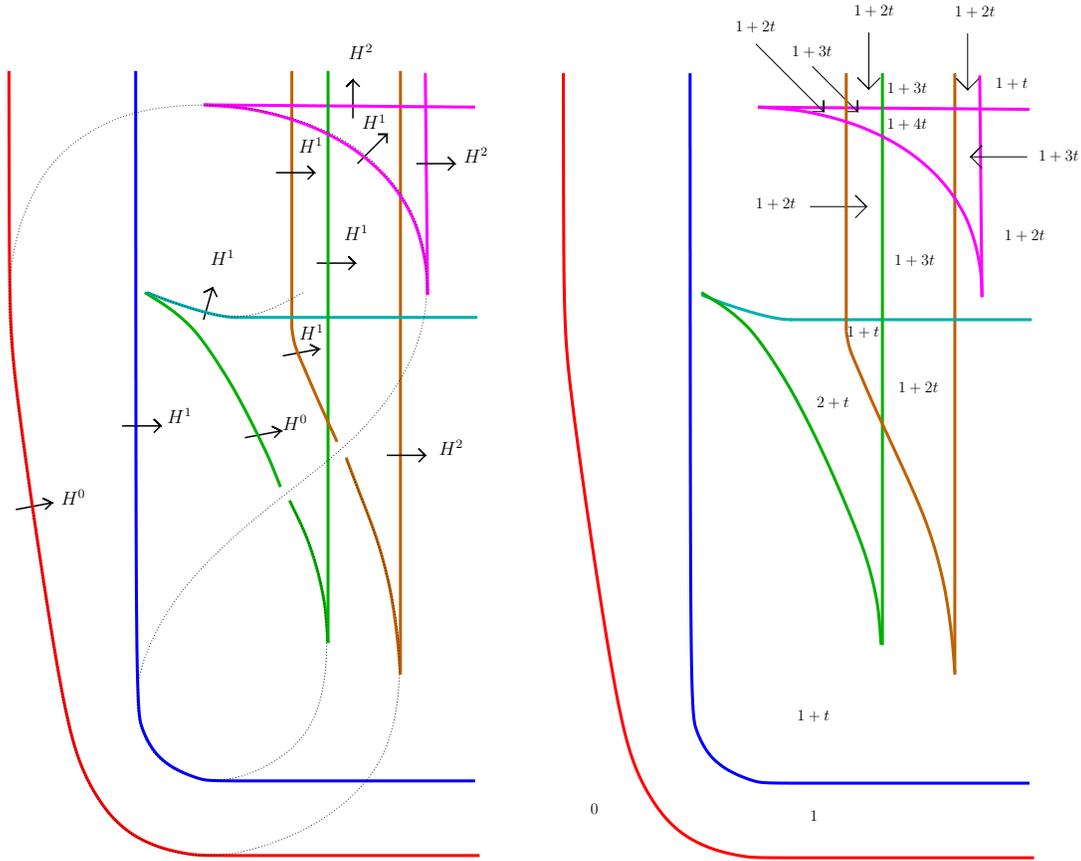}$$
\caption{Klein bottle projection, with Poincar\'e polynomials}\label{kleinproj2}
\end{figure}

In the figure of Example~\ref{spherepocket}, and Figure~\ref{kleinproj2} on the left, we see the Poincar\'e polynomials $P(t,M_{(a,b)})$ for points $(a,b)$ located in regions bounded by Pareto arcs. We are also interested in increments $\Delta P(t,H^j)$ arising as we cross a Pareto arc increasingly in $(a,b)$. A $j$-handle can either create a $j$ generator (new component, closing a hole or a cavity) or kill a $(j-1)$ generator (merging components, filling a hole or a cavity). In the first case, we get $\Delta P(t,H^j)=t^j$, in the second case we get $\Delta P(t,H^j)=-t^{j-1}$. Thus the index of a handle can be read out from $\Delta P$. If it is $t^k$, we have a crating $k$-handle and if it is $-t^k$, we have a cancelling $(k+1)$-handle.

\medskip

The term Pareto arc includes half-lines defined by quadrants $C_{(a,b)}$, Crossing their vertex $(a,b)$ may create
`multiple handles' where $\Delta P$ is not just one term. For example in the vertex of fish tail visible in
Figure~\ref{fig:per-path}, $\Delta P=-t+t^2$. A point at which a single handle is attached will be called {\em generic}.

\medskip

We choose a generic point $(a,b)$ on each Pareto arc and let $H_{(a,b)}$ be the corresponding handle. At this time, the choice is arbitrary but we may want to chose endpoints of an arc, when we study metric sensitive barcodes.

We let $T=[0,1]$ and $R=[r_1,R_1]\times [r_2,R_2]$ be a fixed rectangle in $\Real^2$ containing $f(M)$ in its interior.

\begin{definition}\label{def:per-path}
Let $\{(a_i,b_i)\}_{i=0,1,\ldots m}\subset R$ be a sequence of generic points on Pareto arcs, such that $M_{(a,b)}=\emptyset$ for all $(a,b)$ downward-left of $(a_0,b_0)$, $(a_{i+1},b_{i+1})$ can be reached from $(a_i,b_i)$ going upward-right through the region enclosed by the two arcs, and $M_{(a_m,b_m)}=M$. A {\em persistence path} is a continuous function $\rho:I\to R$ with $\rho(0)=(r_1,r_2)$, $\rho(1)=(R_1,R_2)$ which is non-decreasing in both coordinates, and joins the points of the sequence.
\end{definition}

\begin{figure}
\psfrag{c10}[tl][tl][0.5][0]{$(1,0)$}
\psfrag{c01}[tl][tl][0.5][0]{$(0,1)$}
\psfrag{h0}[tl][tl][0.5][0]{$H^0$}
\psfrag{h1}[tl][tl][0.5][0]{$H^1$}
\psfrag{h2}[tl][tl][0.5][0]{$H^2$}
\psfrag{p0}[tl][tl][0.5][0]{$0$}
\psfrag{p1}[tl][tl][0.5][0]{$1$}
\psfrag{p2}[tl][tl][0.5][0]{$2$}
\psfrag{p1t}[tl][tl][0.5][0]{$1+t$}
\psfrag{p1t2}[tl][tl][0.5][0]{$1+t^2$}
$$\includegraphics[width=12cm]{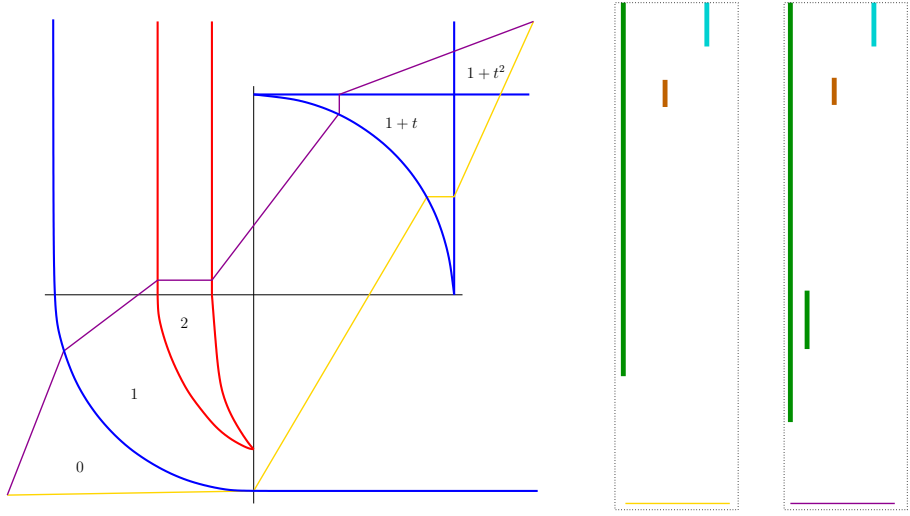}$$
\caption{Left: Two persistence paths for the function in Example~\ref{spherepocket}. Right: Their corresponding barcodes in rectangles marked with the same color as the corresponding path. $\beta_0$ barcodes displayed by green lines, $\beta_1$ by brown lines, and $\beta_2$ by cyan lines.}\label{fig:per-path}
\end{figure}

\medskip

It can be seen that one can find sequences on the arcs so to get piecewise linear persistence (PL) paths with line segments between two consecutive points. This is useful in showing that we get a discrete characterization. For simplicity of notation, we let $H_i=H_{(a_i,b_i)}$ and $M_i=M_{(a_i,b_i)}$. We have a linear filtration $M_0 \subset M_1 \subset \cdots \subset M_m=M$.

\medskip

Figure~\ref{fig:per-path} (left) shows two persistence paths for the cupped sphere presented in Example~\ref{spherepocket}. The path displayed in dark green avoids the pocket, the one in orange passes through it.

\medskip

Note that \cite{Wan} needs a no-cycle condition to apply the generalized gradient. Our persistence paths can be defined even in the presence of Wan's cycles, because they are more restrictive than Wan's admissible curves \cite[Definition 6.3]{Wan}: A persistence path leaves a Pareto arc at the same point as it enters it. Along the path, there may be no cycles, because it is increasing with respect to the poset relation.

\medskip

We now shift our attention to a multidimentional analogy of the Morse inequalities. Our results may be useful for the continuation of the work on the discrete multidimesional Morse-Forman theory initiated in \cite{AlKaLaMa19}.

The following result is an analogy of the Morse equation in the Conley index theory \cite{RybZeh85,Mro91}.

\begin{theorem}[Morse-Conley equation for persistence paths]\label{th:morese-conley-eq}
Let $\rho$ be a persistence path for $((a_i,b_i))_{i=0,1,\ldots m}$ and let $c_j$ be the number of $j$-handles associated to its points. Then there exists a polynomial $Q$ with non-negative integer coefficients such that
\begin{equation}\label{eq:morse-conley}
\sum_{j=0}^{n} c_j\; t^j = P(t,M) + (1+t)Q(t)
\end{equation}
\end{theorem}
\proof The following equation is a direct consequence of the definition of $\Delta P$ and that of persistence path:
\begin{equation}\label{eq:prop-sums}
\sum_{i=0}^{m}\Delta P(t,H_{(a_i,b_i)}) = \sum_{k=0}^{n}\beta_k\; t^k=P(t,M).
\end{equation}
If all handles of $\rho$ create new generators, than, in the light of the preceding discussion, the left-hand side of equation~\ref{eq:prop-sums} is exactly the left-hand side of equation~\ref{eq:morse-conley}. Thus (\ref{eq:morse-conley}) holds with $Q(t)=0$. If a $j$-handle kills a $(j-1)$ generator, then the sum on the left-hand side of (\ref{eq:prop-sums}) misses two terms $t^{j-1}$ and $t^j$ contributing the sum on the left-hand side of (\ref{eq:morse-conley}). By adding all these missing terms to both sides of (\ref{eq:prop-sums}), we get (\ref{eq:morse-conley}) with $Q$ built by terms $t^{j-1}+t^j=(1+t)t^{j-1}$.
\qed

\medskip

By taking $t=-1$ in (\ref{eq:morse-conley}), we get the following corollary.

\begin{corollary}[Euler characteristics]\label{cor:morse-euler}
For any persistence path $\rho$,
\begin{equation}\label{eq:euler}
\sum_{j=0}^{n} (-1)^j \; c_j = \chi(M),
\end{equation}
where $\chi(M)=\sum_{k=0}^{n} (-1)^k \; \beta_k$ is the Euler-Poincar\'e characteristics of $M$.
\end{corollary}

Equation~\ref{eq:euler} is a part of the set of classical Morse inequalities. Since two polynomials are equal is and only if all its coefficients are equal, (\ref{eq:morse-conley}) also gives {\em weak Morse inequalities}~:
\begin{equation}\label{eq:weak-morse}
c_j\geq \beta_j~~ \mbox{ for all } j=1,2,\ldots n.
\end{equation}

\medskip

We conclude this section by deriving a classical result of Morse theory on {\em strong Morse inequalities}. The reader is referred to the book by Milnor~\cite{Mil} for the classical formulation. For the sake of completeness, we present a neat and short proof of an unknown source we have been told about by Marian Mrozek.

\begin{corollary}[Strong Morse inequalities]\label{cor:strong-morse}
For any persistence path $\rho$ and any $k\geq 0$,
\begin{equation}\label{eq:strong-morse}
c_k - c_{k-1} + c_{k-2} + \cdots \pm c_0 \geq \beta_k - \beta_{k-1} +\beta_{k-2} +\cdots \pm \beta_0.
\end{equation}
\end{corollary}
\medskip

\proof 
Knowing that $c_j=\beta_j=0$ for all $j > n$, we can treat Equation~(\ref{eq:morse-conley}) as a power series equation
\begin{equation}\label{eq:morse-conley-pewer-series}
\sum_{j=0}^{\infty} c_j\; t^j = \sum_{j=0}^{\infty}\beta_j\; t^j + (1+t)Q(t).
\end{equation}
Multiplying both sides of (\ref{eq:morse-conley-pewer-series}) by $\Sigma_{i=0}^\infty (-1)^i t^i$, the power series inverse of $(1+t)$, we get
\[
\sum_{k=0}^{\infty} \left( \sum_{i=0}^{k} (-1)^i c_{k-i} \right) t^k = \sum_{k=0}^{\infty} \left( \sum_{i=0}^{k} (-1)^i \beta_{k-i} \right) t^k + Q(t).
\]
Since two power series are equal if and only if all their coefficients are equal and the coefficients of $Q(t)$ are non-negative, we get (\ref{eq:strong-morse}).
\qed

\medskip

Our Morse inequalities should be compared with the work of Wan \cite{Wan}.  Perhaps the main difference between our work
and his, is that we convert functions $f : M \to \Real^2$ into families of filtrations of the manifold $M$.  Wan uses
essentially all of the Pareto arcs to define his filtration, which is often larger than our filtrations.  Moreover he
requires special 'acyclic' Morse 2-functions to even define a filtration of $M$, while any Morse $2$-function works
for us.



\medskip

We now turn our attention to the computability of persistent homology via persistence paths. We associate {\em path-wise barocodes} to any persistence path $\rho$ as follows. First, we want to normalise lengths of persistence paths so to have them all of length $1$. Given a point $(a,b)\in \rho(I)$, let $s(a,b)$ be the euclidean distance from $(r_1,r_2)$ to $(a,b)$ along the path $\rho$ divided by the total length of $\rho$.

\begin{definition}\label{defn:path-barcode}
The {\em $\rho$-barcode} in homology of dimension $k$ is a function on representatives of the $H_k$ generators, whose values are subintervals of $[0,1]$. When a $H_k$ generator is born by a handle attachment at the point $(a_i,b_i)$ and it is killed at the point $(a_j,b_j)$ with $i<j<m$, the corresponding barcode interval is $[s(a_i,b_i),s(a_j,b_j)]$. The {\em life-time} of that generator is $s(a_j,b_j)-s(a_i,b_i)$. If a generator persists until the point $(R_1,R_2)$ of the chosen rectangle, it will also persist if the values of $(R_1,R_2)$ increase. Thus it is reasonable to declare that its life-time is infinite and the corresponding barcode interval is $[s(a_i,b_i),\infty)$.
\end{definition}

Figure~\ref{fig:per-path} (right) shows barcodes of the two persistence paths displayed on the right. It is visible that the life-time of the second generator of $H_0$ created when crossing the pocket is short and it may be null, if we choose the path in dark green that avoids the pocket. Similarly, the life-time of the $H_1$ generator is short.

\medskip

We shall now briefly discuss prospects for numerical implementations of path-wise barcodes. We should emphasize that the aim of our paper is to only provide a theoretical background for computation.

\medskip

Following predecessors \cite{CZ09,CDFFC} who set up computing methods for multi-filtrations, we should consider the family of all piece-wise linear persistence paths $\rho$ built on points $(a_i,b_i)$ in a given finite grid. However, that is a huge family and this choice is likely to lead to computational complexity issues. The size of the family of such paths is most likely similar to that of {\em Young diagrams}~\cite{WF}. Moreover, the number of nodes to join by paths, decisive for the size of the family, increases quadratically with grid subdivisions.

\medskip

For path-wise barcodes, we postulate that it should be sufficient to consider a representable family  $\rep(f)$ of persistence paths built of specific points on Pareto curves: center points, nearly lower-right and upper-left endpoints of Pareto curves, as well as their intersections with horizontal and vertical lines passing through or touching other endpoints. We claim that $\rep(f)$ is a small and exhaustive representation. Moreover, the size of $\rep(f)$ does not increase with grid subdivisions.

\medskip

By {\em exhaustive representation}, we mean here that any additional paths give rise to {\em equivalent barcodes}. That, in turn, means that their barcodes have the same number of intervals for each homology dimension, they may vary by length but appear in the same sequence according to birth and death dates.

\medskip

We are conscious of the fact that, proceeding this way, we are missing the postulate that the persistence should be computed blindly from data, without knowing the exact manifold $M$ and exact function $f$. But it may also be interesting to consider the case when we have $M$ and $f$ given by formulas that enable computing singularities.

\section{Extensions}
\label{sec:ext}

When applying path-wise barcodes to functions which do not satisfy the Wan's no-cycle property \cite[Definitions 6.4 and 6.5]{Wan}, it would be interesting to see what is the information carried by the barcodes of those persistence paths of $\rep(f)$ which cross and go about the cycles of $f$.

\medskip

The filtration of $\Real^2$ by quadrants $C_{(a,b)}$ has a complementary filtration by quadrant
exteriors

$$E_{(a,b)} = \{ (x,y) \in \Real^2 : x \geq a \text{ or } y \geq b \}.$$

Provided the boundary of $C_{(a,b)}$ is transverse to the singular points of $f : M \to \Real^2$,
one has that $f^{-1}(C_{(a,b)})$ is a manifold with corners.  This allows us to use a Poincar\'e Duality isomorphism

$$H_k (f^{-1}(C_{(a,b)})) \simeq H^{m-k} (M, f^{-1}(E_{(a,b)})).$$

Given that quadrant exteriors are the union of three quadrants, this gives a fairly detailed relationship between
the persistent homologies of filtrations corresponding to the four quadrant families:

\[
C_{f_1\leq a,f_2\leq b}=C_{(a,b)},~~C_{f_1\geq a,f_2\leq b},~~ C_{f_1\leq a,f_2\geq b}, ~~ C_{f_1\geq a,f_2\geq b}.
\]

This technique could be thought to be a strong parallel to the theory of trisections of $4$-manifolds
\cite{GK-indefinite, GK-trisecting} as developed by Gay and Kirby.  It also gives a formal setup analogous to
{\em Extended Persistence} of Morse functions, considered in \cite{EH}.


\medskip

Another direction one could take to extrapolate this research would be using smooth functions $M \to \Real^k$ for $k > 2$. This topic is of a great interest to the topological data analysis community. The computational methods of multi-parameter persistent homology such as those in \cite{CZ09,CEFKL,AlKaLaMa19,Les19} are dimension-independent but, on the other hand, they do not have the same insight into the geometry of the encountered singularities as the one we present here for the $\Real^2$-valued functions. There are a variety of useful 'Morse theory' type tools to describe the singularities of functions of this kind.  The analogous theory of multisections of manifolds is developed by Rubinstein and Tillman \cite{RT}.

\medskip

Yet another direction to undertake is the practical implementation of our suggested method for computing path-wise barcodes on the basis of representable family $\rep(f)$.


\end{document}